\documentclass[a4paper,reqno]{amsart}
\usepackage{amssymb,amscd}
\usepackage{color, enumerate}

\definecolor{darkgreen}{rgb}{0.0, 0.5, 0.0}

  \newcommand{\w}{\omega}
  
\newcommand{\fg}{\mathfrak{g}}
  
   \newcommand{\R}{\mathbb{R}}

\newtheorem{proposition}{{\bf Proposition}}

\newtheorem{theorem}{{\bf Theorem}}
\newtheorem{corollary}{{\bf Corollary}}
\newtheorem{definition}{ {\bf Definition}}

\newtheorem{remark}{Remark}

\begin{document}

\author{David Mart\'inez Torres} \address{ Departamento de Matem\'{a}tica
PUC Rio de Janeiro \\ Rua Marqu\^{e}s de S\~{a}o Vicente, 225 \\ G\'{a}vea, Rio de Janeiro - RJ, 22451-900
\\ Brazil}\email{dfmtorres@gmail.com}

\author{Eva Miranda}\address{Department of Mathematics, Universitat Polit\`{e}cnica de Catalunya \\ EPSEB, Edifici P, Avinguda del Doctor Mara\~{n}\'{o}n, 42-44,
 Barcelona, Spain}
\email{eva.miranda@upc.edu}  \thanks{   Eva Miranda is supported by the  Ministerio de Econom\'{\i}a y Competitividad project with reference MTM2015-69135-P.  David Mart\'{\i}nez Torres' research visits to Barcelona concerning this paper were supported by  its precursor MTM2012-38122-C03-01/FEDER }

\title{Weakly Hamiltonian actions}
 \maketitle

 \begin{abstract} In this paper we generalize constructions of non-commutative integrable systems to the context of weakly
 Hamiltonian actions on Poisson manifolds.  In particular we prove that abelian weakly Hamiltonian actions
 on symplectic  manifolds split into Hamiltonian and non-Hamiltonian factors, and explore generalizations in the Poisson setting.

 \end{abstract}

\section{Introduction}

An integrable system on a $2n$-dimensional symplectic manifold  is given by $n$ generically independent pairwise commuting functions.
More generally, a non-commutative integrable system is determined by a set of $2n-r$ integrals ($r\leq n$), out of which $r$ do pairwise commute.
Integrable systems come with \emph{infinitesimal} abelian actions which are Hamiltonian, in the sense that they have  an equivariant \emph{momentum map}.

Furthermore, under compactness assumption on the invariant sets these infinitesimal abelian actions
integrate into a torus action for which there is a normal form (action-angle coordinates).

 However, some discrete integrable systems \cite{suris} do not present commuting first integrals but rather commuting flows. Moreover, there
 are systems that become Hamiltonian after reduction by non-commutative symmetries.
 This justifies considering a more general framework where \emph{weakly Hamiltonian actions} take over Hamiltonian actions.
 In this paper we look at (infinitesimal) actions of abelian Lie algebras on Poisson manifolds having first integrals, but that cannot be arranged into
 an equivariant momentum map.
  Our main purpose is discussing conditions under which one can
   still find
 ``invariant subsets'' (Poisson submanifolds)  where
the residual action is indeed Hamiltonian.

In the symplectic setting there is a already
a result of Souriau in this direction  (\cite{souriau}, Theorem 13.15), which
gives a conceptual explanation of a classical theorem of K\"{o}nig on the decomposition of both the kinetic energy and the motion of a system. However, this
is a result for actions of arbitrary Lie groups (the Galilean group in  K\"{o}nig's theorem) and our main point here is to stress that in the abelian case one can go further
and obtain a splitting not just of the symplectic manifold, but also of the action into a weakly Hamiltonian and a Hamiltonian factor.
Moreover, in the much richer Poisson setting we will see how mild conditions on the action determine Poisson submanifolds,
and how under stronger hypotheses a splitting for the Poisson structure can be found, and in some cases, even for the action. In fact, our global construction --when specialized
to a local setting-- has strong reminiscences of the classical Weinstein local splitting theorem
in Poisson geometry \cite{Weinstein},
and we expect to be able to relate other local splitting results with our construction. Finally our techniques --which are different from Souriau's--
are also appropriate  to draw global conclusions for actions of non-abelian Lie algebras.

\textbf{Acknowledgements:} We are deeply thankful to Alexey Bolsinov for very useful discussions about
the examples in this paper and and for drawing our attention to references \cite{patera,ooms1,ooms2}.
We are also thankful to Yuri Fedorov and Bo\v{z}idar Jovanovi\'{c} for interesting remarks. We are indebted
to the referee for the valuable comments specially those linking this work to the work of Souriau \cite{souriau}.
\section{Motivating examples}
 We start by presenting three different types of (complete) weakly Hamiltonian abelian actions in symplectic and Poisson manifolds, and we close
the section with a related
example.

\subsection{Standard action by translations}\label{sec:standard}

The paradigm of abelian symmetries by Hamiltonian vector fields, but not fitting into a Hamiltonian action, is that of a symplectic vector space $(V,\sigma)$ on which
$V$ itself -- seen as an abelian Lie group -- acts by translations: to each vector $v\in V$ we can assign a first integral which is the unique linear function
 $H_u\in V^*$ such that $dH_u=i_u\sigma$ (or rather, an affine one with that linear part). Since we have:
\[\{H_u,H_v\}=\sigma(u,v)\]
we will never be able to find a basis of first integrals in involution.  Notwithstanding, any such choice of first integrals yields a weakly Hamiltonian action (indeed a non-commutative system with constant brackets).

\subsection{The Galilean group}\label{sec:Galilean}

Let $G(3)$ be the Galilean group and consider its standard  representation (cf. (13.7) in \cite{souriau})
on $T^*\mathbb R^3$ with position and momentum coordinates $q_i$, $p_i$ and particle mass $m$. Recall that $G(3)$ is an extension of the Euclidean group $E(3)$, and  the restriction of this representation to $E(3)$ is Hamiltonian because it is the cotangent lift of its defining action on $\mathbb R^3$. However the Hamiltonian functions corresponding to the Galilean boosts $mq_i$ and translations in the same direction  $p_i$  do not commute; indeed their Poisson bracket is the mass, and the corresponding cocycle is not exact, bringing in another example of weakly Hamiltonian action.

\subsection{Weakly Hamiltonian actions and nilpotent Lie algebras}

The symplectic form on the symplectic vector space $(V,\sigma)$ can be interpreted as a 2-cocycle, and as such it gives rise to $\mathfrak{g}=\R\oplus_\w V$ a central
extension of the abelian algebra $V$. This Heisenberg type-Lie algebra is nilpotent with one dimensional center. The coadjoint
orbit corresponding to the affine hyperplane:
\[\{\alpha\in \mathfrak{g}^*\,|\, \alpha(1,0)=1\}\] can be canonically identified with $V$, and the restriction of the
coadjoint action to this orbit is the linear action by translations above (\ref{sec:standard}).

More generally, let $\mathfrak{g}$ be a nilpotent Lie algebra such that $[\mathfrak{g},\mathfrak{g}]\subset \mathfrak{z}(\mathfrak{g})$ (a 2-step
nilpotent Lie algebra). As brackets lie in the center, they become Casimirs as functions on $\mathfrak{g}^*$, and therefore constants on coadjoint orbits. Then
any subspace of $\mathfrak{g}$ intersecting trivially with the center provides an abelian Lie algebra acting in a weakly Hamiltonian fashion
(which is not Hamiltonian provided that some of the brackets
are non-trivial) on any coadjoint orbit (and in fact on the whole $\mathfrak{g}^*$).
We illustrate this with the following low dimensional example (additional ones can be found by inspecting the list of low dimensional nilpotent
Lie algebras up to dimension 7  (\cite{patera}, \cite{ooms1}, \cite{ooms2})):

The nilpotent Lie algebra of dimension 6, $A_{6,5}^a$ (for $a\neq 0$) for which the non-vanishing
relations on a base (see table III in \cite{patera}) are $[e_1,e_3]=e_5$, $[e_1,e_4]=e_6$, $[e_2,e_3]=ae_6$, $[e_2,e_4]=e_5$.
In this case, the symplectic foliation by coadjoint orbits is given by $e_6^*$ and $e_5^*$, thus defining a foliation with regular
$4$-dimensional symplectic leaves away from zero. The subspace spanned by $e_1,e_2,e_3$ and $e_4$ acts by commuting Hamiltonian vector fields but without
momentum map, providing an example of weakly Hamiltonian action on  a Poisson manifold.

\subsection{Related examples}

In \cite{suris},  motivated by the study of discrete integrable systems,   the ``multi-time'' Legendre transform is applied  to  multi-time Euler-Lagrange equations to obtain a system of commuting Hamiltonian flows.  As observed in \cite{suris} this situation corresponds to having
functions with constant Poisson brackets but these brackets are not necessarily zero\footnote{As observed by the author in \cite{suris} indeed
the vanishing of these Poisson  brackets is equivalent to  Lagrangian 1-form employed in the construction being closed on
the solutions of the Euler-Lagrange equations.}, thus providing an extra motivation to consider weakly Hamiltonian actions.

 \section{Weakly Hamiltonian actions and real analytic functions}

 In this section we discuss how real analytic functions become a tool to study weakly Hamiltonian actions.
\begin{definition}\label{def}  Let $(M,\pi)$ be a  Poisson manifold and let
 \begin{eqnarray*}
\alpha:
\mathfrak{g} & \longrightarrow &\mathrm{poiss}(M,\pi)\\
u & \longmapsto & X_u
\end{eqnarray*}
be an action by Poisson
vector fields.
The action is {\bf weakly Hamiltonian}  if the fundamental vector fields for the action are Hamiltonian vector fields.
The action is called complete if all fundamental vector fields are complete.
  \end{definition}

If $\alpha:\mathfrak{g}\rightarrow \mathrm{ham}(M,\pi)$ is a weakly Hamiltonian action, then it can always be lifted
to a linear map:
\begin{eqnarray*}
 \rho\colon \mathfrak{g}&\longrightarrow & C^\infty(M)\\
  u &\longmapsto & H_u.
\end{eqnarray*}
As it is well-known, the defect from the action  being Hamiltonian is measured by the  Casimir-valued  $2$-cocycle:
\begin{eqnarray}\label{eq:cocycle}\nonumber
c\colon \mathfrak{g}\times\mathfrak{g} &\longrightarrow & \mathrm{Cas}(M)\\
     (u,v) &\longmapsto & \{H_u,H_v\}-H_{[u,v]}.
\end{eqnarray}
More precisely, $c\in \Omega^2_{CE}(\mathfrak{g};\mathrm{Cas}(M))$ is a cocycle in the Chevalley-Eilenberg complex for the trivial
representation of $\mathfrak{g}$
on $\mathrm{Cas}(M)$, which is:
\begin{itemize}
 \item zero if and only if  the chosen lift  defines a Lie algebra morphism (this is equivalent to the mapping $\rho$  being equivariant);
 \item exact if and only if there exists a choice
of lift which is a morphism of Lie algebras.
\end{itemize}

\begin{definition}\label{def:flow}
To any complete weakly Hamiltonian action $\rho\colon \mathfrak{g}\rightarrow C^\infty(M)$, we assign its {\bf flow evaluation map}:
\begin{eqnarray}\label{eq:flow}\nonumber
\zeta\colon \mathfrak{g}\times\mathfrak{g}\times M &\longrightarrow &
C^\infty(\R)\\
     (u,v,x) &\longrightarrow & H_u(\phi^s_v(x)),
\end{eqnarray}
where $\phi^s_v$ denotes the flow of $X_v$.
\end{definition}

Note that if the action $\rho$ is Hamiltonian with momentum map $\mu$, then the flow evaluation map is the result of pulling back via the momentum map
the flow evaluation map for the coadjoint action: indeed, because the momentum map is
$\mathfrak{g}$-equivariant we have
\begin{equation}\label{eq:coadflow2}
\zeta_{u,v,x}(s)=H_u(\phi^s_v(x))=H_u(\phi^s_v\cdot x)=\langle\mathrm{Ad}^*({\Phi^s_v
})(\mu(x)),u\rangle
\end{equation}
where $\Phi^s_v\in G(\fg)$ -- the simply connected Lie group integrating $\fg$ --
and  $\mathrm{Ad}^*({\Phi^s_v})(\mu(x))$ is the corresponding coadjoint  flow of $v$ starting at $\mu(x)$.
Therefore, the $\zeta_{u,v,x}$'s generalize the linear projections of the coadjoint flow  to
the complete weakly Hamiltonian case.

The coadjoint flow is real analytic, and each of its projections (\ref{eq:coadflow2}) is a real analytic function which extends to an entire
function of exponential type (growth). For example, this can be checked by noting that to analyze the coadjoint action it suffices to use matrix groups.
For complete weakly Hamiltonian actions this property still holds (cf. \cite{Hu}, lemma 6,
where real analyticity is studied in the  setting of quasi-representations on symplectic manifolds):

\begin{proposition}\label{pro:main} For any triple $(u,v,x)\in
\mathfrak{g}\times\mathfrak{g}\times M$ the corresponding flow evaluation $\zeta_{u,v,x}$ is real
analytic with expansion at zero:
\begin{equation}\label{eqn:expansion}
\zeta_{u,v,x}(s)=\sum_{j=0}^\infty
H_{\mathrm{ad}^j(v)(u)}(x)\frac{s^j}{j!}+\sum_{j=1}^\infty
c_{\mathrm{ad}^{j-1}(v)(u),v}(x)\frac{s^j}{j!}.
\end{equation}
Moreover, it extends to an entire function of exponential type.
\end{proposition}
\begin{proof}
The fundamental theorem of calculus yields:
\[\begin{split}\zeta_{u,v,x}(s)=\zeta_{u,v,x}(0)+\int_0^s\zeta'_{u,v,x}
(t)dt=\zeta_{u,v,x}(0)+\int_0^s
dH_u(X_v(\phi_v^t(x))dt\\=\zeta_{u,v,x}(0)+\int_0^s\{H_u,H_v\}(\phi_v^t(x))dt=
\zeta_{u,v,x}(0)+\int_0^s
H_{[u,v]}(\phi_v^t(x))dt+c_{u,v}(x)s.
 \end{split}\]
The formula in (\ref{eqn:expansion}) follows by induction. In order to check
convergence of the power series expansion fix any metric in $\mathfrak{g}$, and
pick a neighborhood of the origin $W\subset \mathfrak{g}$ so that
\[|[u,v]|\leq C|u||v|,\, \forall \in u,v\in W.\]
The linear maps  $u\mapsto H_u(y)$ are continuous, and also vary continuously
with $y$.  The bilinear map $(u,v)\mapsto c_{u,v}(x)$ is also continuous. Therefore
the norm of the remainder of the Taylor expansion can be bounded in any compact
subset as follows:
\[\begin{split}\left|\int_0^{s_1}\cdots
\int_0^{s_{j+1}}H_{\mathrm{ad}^j(v)(u)}(\phi^{s_{j+1}}(c))ds_{j+1}\cdots ds_1 +
c_{\mathrm{ad}^{j+1}(v)(u),v}(x)\frac{s^{j+1}}{(j+1)!}\right|\leq\\
K_1|u|\frac{(C|v|s)^{j+1}}{(j+1)!}+K_2|u|\frac{(C|v|s)^{j+1}}{(j+1)!}\end{split}
\]
Then uniform convergence on any compact subset is straightforward.

Nearly the same proof shows  that the power expansion has
coefficients whose norm is dominated by those of an exponential, and therefore
the desired result follows.

\end{proof}

One can roughly rephrase Proposition \ref{pro:main}  by saying that a complete
weakly Hamiltonian action on a Poisson manifold gives a large supply of very
special analytic functions (non-trivial except for abelian Hamiltonian actions).

\begin{remark} If $\rho$ is a complete weakly Hamiltonian action on a symplectic manifold,
then the momentum map becomes $G(\mathfrak{g})$-equivariant for the (affine) correction
of the coadjoint action by the integration of the cocycle (\ref{eq:cocycle}) (\cite{souriau}, \cite{LM,marle} and \cite{houard}). The flow evaluation map becomes again
the pullback of the linear projections for this
corrected action, for which formula (\ref{eqn:expansion}) can readily be obtained.
Indeed it can be  easily checked that in the symplectic setting the following formula is obtained for the flow evaluation map where $\theta$ is the 1-cocycle associated to the action,
$$\zeta_{u,v,x}(s)=\left< \mu(x), Ad^*_{\exp(-sv)}u \right>+ \left< \theta(\exp(sv)),u \right>.$$
 We chose not to ``embed" weakly Hamiltonian
actions into the equivariant framework
both to make this note self-contained and
because this cannot be done in general for Poisson manifolds.
\end{remark}

\begin{remark} In \cite{Ku} it is shown that any action of a compact group on a symplectic
manifold is equivalent to an analytic one (also the symplectic form being
analytic). Thus if the action is weakly Hamiltonian one obtains automatically
that the functions $\zeta_{u,v,x}$ are analytic.  Our result for symplectic manifolds is more general in
the sense that it only requires a complete Lie algebra action, and it makes
clear the dependence of the expansion of $\zeta_{u,v,x}$ in terms of the Lie
algebra structure and the cocycle $c$.
\end{remark}

\section{ Splitting of the structures}
 K\"{o}nig's theorem \cite{konig} establishes that the  kinetic energy of a system  is the sum of the kinetic energy associated to
the movement of the center of mass and the kinetic energy associated to the movement of the particles relative to the center of mass ( the so-called internal energy).
In particular, the motion of a particle can be decomposed as the motion of its center of mass (which by the conservation of linear momentum moves
uniformly in a line) and the motion of the system around its center of mass.  K\"{o}nig's theorem is interpreted in \cite{souriau}, paragraph 13.25,
as a decomposition theorem associated
to the normal abelian subgroup of the Galilean group corresponding to fixing the identity matrix in $\mathrm{SO}(3,\mathbb R)$ in its matrix representation.

In this section we generalize Souriau's decomposition theorem in several directions. On the one hand, we discuss actions on Poisson manifolds. On the other hand,
because we only discuss abelian Lie algebras, we can find sufficient conditions to split also the action into a translational standard factor and a Hamiltonian factor.
This makes precise the idea
that in some systems after reducing by translations (or restricting to appropriate slices to the translation action), what we get is a Hamiltonian system (and sometimes
a completely integrable one).

\subsection{The Poisson case}

Let $\rho:\mathfrak{g}\to C^\infty(M)$ be a weakly Hamiltonian action with cocycle $c$. Let us consider $\mathfrak{g}\times M\to M$
as a trivial vector bundle over $M$. For simplicity, let us assume that $M$ is connected. A Poisson manifold is naturally endowed with a symplectic foliation $\mathcal F$ whose tangent distribution is generated by the set of all Hamiltonian vector fields. As such it has an associated leaf space $M/\mathcal F$.

\begin{definition} We say that $F\to M$ is a sub-bundle of $\mathfrak{g}\times M$ over the leaf space if $F\to M$ is a sub-bundle such that
 the fibers $F_x$ and $F_y$ are equal  whenever $x,y\in M$ are in the same symplectic leaf.

A framing for a sub-bundle over the leaf space is a global trivialization $\mathfrak{b}:F\to \R^l$ which is constant over the symplectic leaves.
\end{definition}

We are interested in sub-bundles $F\to M$ over the leaf space which are symplectic relative to the cocycle $c\in \Omega^2_{CE}(\mathfrak{g};\mathrm{Cas}(M))$. By this we mean that for all $x\in M$
the constant 2-form $c(x)$ restricts  to a symplectic form along $F_x$.

Such sub-bundles give rise to an interesting structure on $(M,\pi)$:
\begin{theorem}\label{thm:1}  Let $\rho\colon \fg\rightarrow C^\infty(M)$ be
a complete weakly Hamiltonian action of an abelian Lie algebra and let $c$ be its associated cocycle. If $ F\to M$
is a vector sub-bundle over
the leaf space
 which is symplectic relative to $c$, then it determines:
 \begin{enumerate}
  \item A Poisson-Dirac submanifold $N$.
  \item A diffeomorphism
  \[\Phi: M\to F|_N\]
  such that the inverse image of the fiber $F_x$, $x\in N$,  is a symplectic submanifold $L_x$ of the leaf through $x$ and
  \[\Phi: (L_x,\pi^{-1}|_{L_x})\to (F_x,c(x))\]
  is a symplectomorphism.
  \end{enumerate}
  Moreover, if $F\to M$ has a framing $\mathfrak{b}$ over the leaf space, then the diffeomorphism has target a product manifold (trivial bundle)
  \begin{equation}\label{eq:split1}
   \Phi\to M\rightarrow \R^{2d}\times N,
  \end{equation}
 and
   $\Phi^{-1}(N\times \{u\})$, $u\in \R^{2d}$, is a Poisson-Dirac submanifold of $(M,\pi)$.

\end{theorem}

\begin{proof}
We define
\[N=\{x\in M\,|\, H_u(x)=0,\, \forall u\in F_x\}.\]
If $x\in N$, we  fix $\mathfrak{b}$ a local framing around $x$ which is constant on the leaf space. Note that this is always possible as
we can start with $v_1,\dots,v_{2d}$ a basis of $F_x$ and then project orthogonally into nearby fibers used a fixed
inner product on $\mathfrak{g}$.

Using this framing we define the map:
\[\psi_\mathfrak{b}:U\subset M\to \R^{2d},\quad x\mapsto (H_{\mathfrak{b}^{-1}(e_1)}(x),\dots,H_{\mathfrak{b}^{-1}(e_{2d})}(x)),\]
where $e_1,\dots,e_{2d}$ is the canonical basis.
By linearity of the representation $\rho$,  $N\cap U=\psi_\mathfrak{b}(0)$, so it suffices to show that 0 is a regular value.

For $u\in F_x$, we pick $v\in F_x$ so that $c_{u,v}(x)\neq 0$. Because
\[H_u(X_v)=\w(X_u,X_v)=\{H_u,H_v\}=\{H_u,H_v\}-H_{[u,v]}=c_{u,v},\]
$\psi_\mathfrak{b}$ is a submersion restricted to each symplectic leaf.

We also need to prove that $N$ is non-empty. Given $y\in M$, we look for $v\in F_y$ such that $\phi_v^1(y)\in N$. This
is equivalent to:
\begin{equation}\label{eq:system}
0=H_u(\phi_v^1(y))=\zeta_{u,v,y}(1),\, \forall u\in F_y.
 \end{equation}

As in  the abelian case the formula for the flow evaluation (\ref{eqn:expansion})
becomes
\begin{equation}\label{eqn:ab}
\zeta_{u,v,y}(s)=H_u(y)+c_{u,v}(y)s,
\end{equation}
and formula (\ref{eq:system}) becomes the linear system of equations:
\[H_u(y)+c_{u,v}(y)=0\]
By non-degeneracy of $c_{u,v}(y)$, there exist a unique solution of this system. Therefore $N$ is a non-empty submanifold  of $M$ intersecting each leaf transversely.

Now  we define the map:
\[\Phi:M\to F|_N ,\quad y\mapsto (v,\phi^1_v(y)),\]
where $v$ is the unique solution of (\ref{eq:system}). This map is clearly a bijection. It is smooth because for $z$ in a small neighborhood of $y$, the point
$\phi_v^1(z)$ belongs to a small neighbourhood of $\phi^1_v(y)$, where the local trivialization (over the leaf space) can be used
to correct $v$ into the appropriate vector $v(z)$ (and the system we have to solve varies smoothly with the point $z$).  The smoothness of the inverse
\[\Phi^{-1}(u,x)=\phi_{-u}^1(x)\] is clear.

It remains to prove Poisson theoretic properties of $\Phi$. The foliation defined by the fibers of $F|_N$ is transferred by $\Phi$ to a foliation
with leaves $L_x$, $x\in N$. This leaf can be alternatively described as the orbit
through $x$ of the action of $F_x\subset \mathfrak{g}$. In particular, for any $u\in F_x$ the Hamiltonian vector field $X_u|_{L_x}$ is tangent to $L_x$,
and a basis of $F_x$ provides a framing of the tangent bundle $TL_x$. By construction at $y\in L_x$:
\[\pi^{-1}(X_u,X_v)=c_{u,v}(x),\]
thus proving item (2).

Recall from \cite{CF2} that $N\subset (M,\pi)$ is a Poisson-Dirac submanifold if
\begin{equation}\label{eq:dirac}
 T_xN\cap \pi^{\#}(T_xN^0)=\{0\},
\end{equation}
and the induced bi-vector $\pi_N\in \mathfrak{X}^2(N)$ happens to be smooth.

We showed that $T_xN$ is the common kernel of $\{dH_u\,|\,u\in F_x\}$, or, conversely, that the latter subspace of covectors is the annihilator
$T_xN^0$.

As $\pi^\#(dH_u)=X_u\in T_xL_x$, equation (\ref{eq:dirac}) holds.

The smoothness of $\pi_N$ follows from the smoothness
of the distribution defined by the symplectic leaves $L_x$, $x\in N$, because its annihilator at points of $N$ can be identified with $T^*N$,
and $\pi_N^\#$ is the restriction of $\pi^\#$ to this smooth sub-bundle \cite{CF2}.

Finally, if we have a framing $\mathfrak{b}$ of $F\to N$ over the leaf space, then $F$ becomes trivial and we get a diffeomorphism:
\[\Phi: M\to \R^{2d}\times N.\]
The same arguments as above show that $\Phi^{-1}(\{u\}\times N)=\psi_\mathfrak{b}^{-1}(u)$ is a Poisson-Dirac submanifold of $(M,\pi)$.
\end{proof}

If $F\to M$ is a framed sub-bundle over the leaf space and it is symplectic relative to the cocycle $c$, there is no reason why the
framing should trivialize as well the fiberwise linear symplectic forms. If that is the case, then $(M,\pi)$ is in fact a product Poisson manifold:
\begin{theorem}\label{thm:prod2} Let $\rho\colon \fg\rightarrow C^\infty(M)$ be
a complete weakly Hamiltonian action of an abelian Lie algebra and let $c$ be its associated cocycle. If $ F\to M$
is a vector sub-bundle over
the leaf space
 which is symplectic relative to $c$, then a framing over the leaf space compatible with the fiberwise symplectic structure determines a splitting
 of $M$ into a symplectic vector space and a Poisson submanifold:
 \[\Phi:(M,\pi)\to (\R^{2d}\times N, \w^{-1}\oplus \pi|_ N).\]
 \end{theorem}
\begin{proof}
By Theorem \ref{thm:1} we have a diffeomorphism
\begin{equation}\label{eq:prod2}
\Phi:(M,\pi)\to \R^{2d}\times N.
 \end{equation}

A vector $u\in \R^{2d}$ acts on the right hand side of  (\ref{eq:prod2}) by translations of the first factor. The diffeomorphism $\Phi$ conjugates
this translation to the time one flow of the the Hamiltonian vector field
of the function $H_{\mathfrak{b}^{-1}(u)}$. This, together with the fact that the function by hypothesis ${\Phi|_{L_x}}_*c(x)$
is a constant symplectic form $\w$ on $\R^{2d}$, implies that $\Phi$ sends $\pi$ into the product Poisson structure $ \w^{-1}\oplus \pi|_ N$.

\end{proof}

\begin{remark} A special instance where Theorem \ref{thm:prod2} applies is when
there exists $V$ a subspace of $\mathfrak{g}$ which intersects trivially all $\{\ker c|_F\}_{F\in \mathcal{F}_\pi}$. We obtain
a Poisson splitting:
\[(M,\pi)\cong (V\times N, {c|_V}^{-1}\times \pi|_N),\]
given by the action of $V$.

It is worth
pointing out that this  splitting result  can be seen as a global version of
Weinstein splitting theorem. The original proof of the Weinstein's splitting theorem
amounts to constructing a weakly Hamiltonian abelian action of maximal rank with
symplectic cocycle.

It is also possible to reinterpret other splitting results in the Poisson setting using this language. For
example, in \cite{MZ,FM} an equivariant Weinstein splitting
theorem is proved at a fixed point of a Poisson for an action of a compact Lie group.
For a Hamiltonian
action of compact $G$, this result can be restated saying that the (local)
Hamiltonian action of $\fg$ can be extended to a weakly  Hamiltonian action of
an extension $\fg \ltimes \R^{2d}$, where $2d$ is the rank of the Poisson structure at $x$
(the cocycle has kernel $\fg$). From this perspective, the equivariant  Poisson splitting
 follows trivially.

\end{remark}

As for induced actions, the sub-bundle $F\to M$ in Theorem \ref{thm:1} can be regarded as a groupoid, and as such it acts on $M$. As we prefer not
to go to the greater generality of groupoid actions on Poisson manifolds, we confine ourselves to the situation in which the sub-bundle is framed.
This produces an action of $\R^{2d}$ on $(M,\pi)$ with cocycle the pullback of $c|_F$ under the framing (very far from being Hamiltonian). This is the action
we used in the proof of Theorem \ref{thm:prod2}.

We are interested in sufficient conditions for the existence of \emph{``invariant Hamiltonian submanifolds"}:

\begin{theorem}\label{thm:splitpoisson2} Let $\rho\colon \fg\rightarrow C^\infty(M)$ be
a complete weakly Hamiltonian action of an abelian Lie algebra and let $c$ be its associated cocycle.  If
$\ker c\to M$ admits a framing $\mathfrak{b}'$ over the leaf space, then an inner product on $\mathfrak{g}$ determines
a Poisson-Dirac submanifold $N$ and a residual Hamiltonian action
\[\rho_{\mathfrak{b}'}:\R^k\to C^\infty(N)\]
 \end{theorem}
\begin{proof}
Because $\ker c\to M$ is a bundle over the leaf space, so is $\ker c^\perp\to M$. By applying Theorem \ref{thm:1} to
$\ker c^\perp\to M$ we obtain a Poisson-Dirac submanifold $N\subset (M,\pi)$. The action on $N$
is induced by the framing $\mathfrak{b}'$ exactly as in the proof of Theorem \ref{thm:prod2}. As the corresponding cocycle is the restriction of $c$ to $\ker c\to N$,
 it follows that the action is Hamiltonian.
\end{proof}

\subsection{ The symplectic case}
As a consequence of the theorem above applied in the symplectic context, a slightly improved version of Souriau's splitting theorem is obtained and  the action can be split into a Hamiltonian factor and a weakly Hamiltonian action corresponding to the first motivating example in this paper:

\begin{corollary}\label{cor:1} Let $\rho\colon \fg\rightarrow C^\infty(M)$ be
a complete weakly Hamiltonian action of an abelian Lie algebra on a symplectic manifold and let $c\in \wedge^2\mathfrak{g}^*$ be its associated cocycle.
Any subspace
$V\subset \mathfrak{g}$ complementary to the kernel of $c$ determines a symplectic splitting
\[\Phi:(M,\w)\to (V\times N,c|_V\oplus \w|_N)\]
and a splitting of the action into a standard translational action on $(V,c|_V)$ and a Hamiltonian one on $(N,\w|_N)$.
\end{corollary}

In the symplectic setting it is useful to consider
 the correction of the weakly Hamiltonian action by the 1-cocyle. Denote by $\mu:M\longrightarrow \mathfrak{g}^*$ the associated momentum map.
In this case
the symplectic manifold $N$ can be interpreted as the Marsden-Weinstein reduction associated to the $V$-action with moment map $\tilde{\mu}=\pi\circ\mu$ where $\pi$ is the projection $i^*$ with $i:V\hookrightarrow \mathfrak g$ .  This identification can be done by considering the pre-image of any element in $V^*$
by $\tilde\mu$ which by construction intersects each orbit of the action at a single point
therefore its quotient by the action of $V$ is $N$.

 Corollary \ref{cor:1} is a  generalization of a (local) result
in \cite{RW71}. There, one assumes the existence of what is called a
semicanonical system of functions in a 2n dimensional symplectic manifold. In
our language these is a weakly Hamiltonian action of an n-dimensional abelian
Lie algebra, so that the differentials of the associated Hamiltonian functions
are linearly independent. It turns out that neither hypothesis on the number of
functions nor on the rank of differentials of the functions in involution are necessary.

\begin{remark}If  $(M,\pi)$ is a Poisson manifold supporting only trivial Casimirs (constants), then
the symplectic splitting theorem holds word by word. As examples of Poisson manifolds with trivial  Casimirs we may consider, for instance,  the Reeb foliation of $S^3$ with
leafwise area form, compact cosymplectic manifolds with non-compact leaves endowed with  natural Poisson structures \cite{guimipi}, and other Poisson manifolds constructed out of them via products, surgeries, etc.
\end{remark}

\section{An application to nilpotent actions}
We end up this paper with an application of Proposition \ref{pro:main}  to  study of nilpotent actions which we consider only in the symplectic context for the sake of simplicity.

The real analyticity of the flow evaluation map is rather powerful, and it has been used  (under a different guise) to draw consequences
on complete actions of nilpotent and semisimple Lie algebras on compact manifolds \cite{De}.

Here we want to point out yet another global results for nilpotent actions:

\begin{corollary}  Let $\rho\colon \fg\rightarrow C^\infty(M)$ be a complete
weakly Hamiltonian effective representation  of a nilpotent non-abelian Lie
algebra on a (non-compact) symplectic manifold. Let $v$ be a vector not in the center of
$\fg$. Then the any periodic of $X_v$ must be contained in the level set of a
non-constant function (in general different from $H_v$).

\end{corollary}
\begin{proof} By our assumptions we can find $u\in \fg$ such that according to (\ref{eqn:expansion}) for all $x\in
M$,  $\zeta_{u,v,x}(s)$  is a polynomial with linear coefficient
$H_{[v,u]}(x)-c_{u,v}$, which is a non-constant function on $x$. Hence if we have a periodic
orbit of $X_v$, by compactness $\zeta_{u,v,x}(s)$ must be constant and therefore  must be in the zero subset of $H_{[v,u]}(x)-c_{u,v}$
(and also  in the zero set of functions corresponding to  the coefficients of higher order in (\ref{eqn:expansion})).
\end{proof}

As for an analog of the splitting theorem for abelian complete  weakly
Hamiltonian actions, the situation in the nilpotent case is much more
complicated. One may take a basis of $\fg$ and use the corresponding
Hamiltonians to arrange a function to Euclidean space, and one can obtain the following information:
for a given point $x$ and
$u,v\in \fg$:
\begin{itemize}
\item either $H_{\rm{ad}^{j}v(u)}(x)-c_{\rm{ad}^{j-1}v(u),v}=0$ for all $j$, in which case
the orbit of $X_v$ through $x$ will either not intersect $H_u^{-1}(a)$, $a\in
\R$, or be contained in it;
\item or some  $H_{\rm{ad}^{l}v(u)}(x)-c_{\rm{ad}^{l-1}v(u),v}\neq 0$, in which case for
all $a\in \R$ the orbit of $X_v$ through $x$ intersects $H^{-1}_u(a)$ a finite
number of times.
\end{itemize}


\begin{thebibliography}{xxxxx}

\bibitem{CF2} M. Crainic, R. L. Fernandes,  \emph{Integrability of Poisson brackets.} J. Differential Geom. 66 (2004), no. 1, 71--137.


\bibitem{De} T. Delzant, \emph{ Sous-alg\`ebres de dimension finie de l'alg\`ebre
des champs Hamiltonienes.} Preprint available on the author's webpage.


\bibitem{FM} P. Frejlich and  I. Marcut, \emph{The Normal Form Theorem around Poisson Transversals}, arXiv:1306.6055.


\bibitem{guimipi}
V. Guillemin, E. Miranda, and A.   Pires,
\newblock {\emph{Codimension one symplectic foliations and regular Poisson structures.}}
\newblock { Bulletin of the Brazilian Mathematical Society}, 42(4):607--623,
  2011.

 \bibitem{houard} J.C. Houard, \emph{An integral formula for cocycles of Lie groups}, Ann. Inst. Henri Poincar\'{e}
Physique th\'{e}orique (1980), 221–247.

\bibitem{Hu} V. Humili\`ere,  \emph{Hamiltonian pseudo-representations.}  Comment.
Math. Helv.  84  (2009),  no. 3, 571--585.

\bibitem{konig} S. K\"{o}nig, \emph{ De universali principio æquilibrii \& motus, in vi viva reperto, deque nexu inter vim vivam \& actionem, utriusque minimo, dissertatio}, Nova acta eruditorum (1751) 125-135, 162-176.

\bibitem{Ku} F. Kutzschebauch and F. Loose,  \emph{Real analytic structures on a
symplectic manifold.}  Proc. Amer. Math. Soc.  128  (2000),  no. 10, 3009--3016



\bibitem{camilleevapol}
C.  Laurent-Gengoux, E.  Miranda, and P.  Vanhaecke.
\newblock {\em Action-angle coordinates for integrable systems on Poisson manifolds.}
\newblock   Int. Math. Res. Not. IMRN 2011, no. 8, 1839--1869.


\bibitem{LM} P. Libermann and C.M. Marle, Symplectic geometry and analytical mechanics. Translated from the French by Bertram Eugene Schwarzbach. Mathematics and its Applications, 35. D. Reidel Publishing Co., Dordrecht, 1987.


\bibitem{marle}  C.M. Marle, \emph{Symmetries of
Hamiltonian systems on Symplectic and Poisson manifolds}, Similarity and symmetry methods, 185–269,
Lect. Notes Appl. Comput. Mech., 73, Springer, Cham, 2014.

\bibitem{MZ} E. Miranda and N. T. Zung, \emph{ A note on equivariant normal forms
of Poisson structures.}  Math. Res. Lett.  13  (2006),  no. 5-6, 1001--1012.



\bibitem{patera} J. Patera, R. T. Sharp, P. Winternitz and H. Zassenhaus,
\emph{Invariants of real low dimension Lie algebras.}
J. Mathematical Phys. 17 (1976), no. 6, 986--994.


\bibitem{ooms1} A. Ooms,
\emph{The Poisson center and polynomial, maximal Poisson commutative subalgebras, especially for nilpotent Lie algebras of dimension at most seven.}
J. Algebra 365 (2012), 83--113.

\bibitem{ooms2} A. Ooms,
\emph{Computing invariants and semi-invariants by means of Frobenius Lie algebras.}
J. Algebra 321 (2009), no. 4, 1293--1312.

\bibitem{RW71} J. Roels and A.  Weinstein,  \emph{Functions whose Poisson brackets
are constants.}  J. Mathematical Phys.  12  1971 1482--1486.

\bibitem{souriau} J.-M., Souriau, Structure des syst\`emes dynamiques, Dunod, Paris 1969.

\bibitem{suris} Y. Suris,
\emph{Variational formulation of commuting Hamiltonian flows: multi-time Lagrangian 1-forms.}
J. Geom. Mech. 5 (2013), no. 3, 365--379.


\bibitem{Weinstein} A. Weinstein, \emph{The local structure of Poisson manifolds.},
J. Differential Geom. 18 (1983), no. 3, 523--557.
\end{thebibliography}
\end{document}